\documentstyle[11pt,amssymb,amsthm,amsmath,graphicx]{article}

\newtheorem{thm}{Theorem}

\newtheorem{cor}[thm]{Corollary}

\newtheorem{rem}[thm]{Remark}
\newtheorem{ex}[thm]{Example}

\def\blfootnote{\xdef\@thefnmark{}\@footnotetext}

\title{Cyclic extensions of Moufang loops induced by semi-automorphisms}
\date{}
\author{STEPHEN M GAGOLA III 
\\ \textit{Department of Algebra, Charles University,}\\ \textit{Sokolovsk\'{a} 83, 186 75 Prague 8, Czech Republic} \\ \textit{gagola@karlin.mff.cuni.cz}}

\begin{document}

\maketitle  

\begin{abstract}
\noindent It is well known that if a group $G$ factorizes as $G = NH$~where 
$H\leq G$ and $N\trianglelefteq G$ then the group structure of $G$ is determined~by~the subgroups $H$ and $N$, the intersection $N\cap H$ and how $H$~acts~on~$N$~with~a homomorphism $\varphi : H \rightarrow$ Aut$(N)$.
Here we generalize the idea by creating extensions using the semi-automorphism group of $N$.
We show that if $G=NH$ is a Moufang loop, $N$ is a normal subloop, and $H=\left<u\right>$ is a finite cyclic group of order coprime to three then the binary operation of $G$ depends only on the binary operation~of~$N$,~the intersection $N\cap H$, and how $u$ permutes the elements of $N$ as a semi-automorphism~of~$N$. 
\end{abstract}


\small
${}$
\normalsize

\vspace{.1in}

\begin{center}
${}$\\ 1. I{\footnotesize NTRODUCTION}
\end{center}

The concept of group extensions with products of groups is well known.  In particular, semidirect product decompositions have been extremely useful in studying the structure of groups.
If $G$ is a group with a normal subgroup $N$ then $G$ is said to be an \it extension \rm of $G/N$ by $N$. The ``extension problem'' arises when
the groups $N$ and $G/N$ are known and properties of $G$ are to be determined. 
Such a 
problem can be generalized by looking at how extensions can be induced by semi-automorphisms of $N$ rather than automorphisms~of~$N$. This can be viewed in the light of the following theorem. 

\begin{thm}\label{T:main}  Suppose $G$ is a Moufang loop with $G=NH$ where $H\leq G$ and $N\trianglelefteq G$. If $H=\left<u\right>$ is a finite cyclic subgroup of order coprime to three then for any $(xu^m),(yu^n)\in G$ with $x,y\in N$ 
\begin{equation}\label{e:ext} (xu^m)(yu^n)= f^{ \frac{2m+n}{3}}\!\!\left(f^{\frac{-2m-n}{3}}(x)f^{\frac{m-n}{3}}(y)\right)u^{m+n} \end{equation}
where 
\begin{alignat*}{3}
f:\: &N &&\longrightarrow N \\
   &g &&\longmapsto ugu^{-1}
\end{alignat*}
is a semi-automorphism of $N$.  Moreover, $G$ is a group if and only if $N$ is a group and $f$ is an automorphism of $N$.
\end{thm}

Just like for groups, such a result can be extremely useful in studying the structure of Moufang loops.
Likewise, this could 
be an 
efficient piece of machinery when solving open problems and classifying Moufang loops of small order.
For example, Theorem $\ref{T:main}$ has already been used to partially prove an open problem proposed by J.D. Phillips showing that every Moufang loop with an order coprime to six has a nontrivial nucleus $\cite{nucleus}$.  Moreover, Theorem \ref{T:main} has been used in \cite{inner} proving that any Moufang loop that can be generated by its cubes has an inner mapping group that can be generated by conjugation maps.  From this it follows that a subloop is normal if and only if it is stabilized by conjugation maps.


Theorem \ref{T:main} also increases the significance of the general study of semi-automorphisms.  Semi-automorphisms of groups were studied by \mbox{F. Dinkines,} I.N. Herstein and M.F.~Ruchte back in the 1950s.
\mbox{F. Dinkines} $\cite{Dinkines}$ showed that a semi-automorphism of either a symmetric group or an alternating group is either an automorphism or an anti-automorphism. 
It was then I.N.~Herstein and M.F. Ruchte $\cite{Herstein}$ who proved that if $G$ is a simple group that contains an element of order 4
then any semi-automorphism of $G$ is either an automorphism or an anti-automorphism.
F.~Dinkines, I.N.~Herstein and M.F. Ruchte also conjectured that any semi-automorphism of a simple group would be either an automorphism or an anti-automorphism.

\begin{center}
${}$\\ 2. M{\footnotesize AIN} R{\footnotesize ESULT}
\end{center}

A \it Moufang loop \rm is a generalization of a group that arises when the associative law is replaced by any of the following equivalent identities:
\vspace{-.03in}
\begin{alignat*}{5}
 (xy)(zx) &= x((yz)x), \:\:
 &&(xy)(zx) &&= (x(yz))x,\\
 ((xy)x)z &= x(y(xz)),
 &&((zx)y)x &&= z(x(yx)).
\end{alignat*}
In particular, by Moufang's Theorem, these loops are \it diassociative \rm in the sense that any subloop generated by two elements is associative and therefore a group. 
A \it semi-automorphism \rm of a Moufang loop $G$ is a bijection $f$ of $G$ such that  $f(aba)=f(a)f(b)f(a)$ for any $a,b\in G$. 
The 
set of semi-automorphisms of $G$ is a group under composition and is denoted by SemiAut$(G)$. \\

\noindent \it Proof of Theorem \ref{T:main}. \rm 

It is known by R.H.~Bruck~$\cite{pseudo}$ that any inner map of $G$ is a pseudo-automorphism, and therefore a semi-automorphism of $G$. Thus the map 
\begin{alignat*}{3}
f:\: &N &&\longrightarrow N \\
     &g &&\longmapsto  ugu^{-1}
\end{alignat*}
is a semi-automorphism of $N$.  Furthermore, since the order of $u$ is coprime to three, there exists a unique element of $H$, denoted by $u$\footnotesize$^{1/3}$\normalsize, such that its cube is $u$.  Likewise there exists a cubed root of $f$ in $\left<f\right>$, denoted by $f$\footnotesize$^{1/3}$\normalsize, which maps $g$ 
to $u$\footnotesize$^{1/3}$\normalsize$gu$\footnotesize$^{-1/3}$ \normalsize in $N$. Now since $G$ is a Moufang loop and therefore satisfies the Moufang identities,
\begin{alignat*}{1}
(xu^m)(yu^n) &=  \left(\left(x\left(u^{\frac{2m+n}{3}}u^{\frac{m-n}{3}}\right)\right)\left(yu^n\right)\right)\left(u^{\frac{m-n}{3}}u^{\frac{-m+n}{3}}\right)  \\
&= \left(\left(\left(\left(xu^{\frac{2m+n}{3}}\right)u^{\frac{m-n}{3}}\right)\left(yu^n\right)\right)u^{\frac{m-n}{3}}\right)u^{\frac{-m+n}{3}}  \\
             &= \left(\left(xu^{\frac{2m+n}{3}}\right)\left(u^{\frac{m-n}{3}}\left(yu^n\right)u^{\frac{m-n}{3}}\right)\right)u^{\frac{-m+n}{3}} \\
             &= \left[\left(u^{\frac{2m+n}{3}}\!\!\left(u^{\frac{-2m-n}{3}}xu^{\frac{2m+n}{3}}\right)\right)\left(\left(u^{\frac{m-n}{3}}yu^{\frac{-m+n}{3}}\!\!\right)u^{\frac{2m+n}{3}}\right)\right]u^{\frac{-m+n}{3}} \\
             &= \left[u^{\frac{2m+n}{3}}\left(\left(u^{\frac{-2m-n}{3}}xu^{\frac{2m+n}{3}}\right)\!\!\left(u^{\frac{m-n}{3}}yu^{\frac{-m+n}{3}}\right)\right)u^{\frac{2m+n}{3}}\right]u^{\frac{-m+n}{3}} \\
&= u^{\frac{2m+n}{3}}\left(\left(u^{\frac{-2m-n}{3}}xu^{\frac{2m+n}{3}}\right)\!\!\left(u^{\frac{m-n}{3}}yu^{\frac{-m+n}{3}}\right)\right)\left(u^{\frac{-2m-n}{3}}u^{m+n}\right) \\
             &= \left(u^{\frac{2m+n}{3}}\left(\left(u^{\frac{-2m-n}{3}}xu^{\frac{2m+n}{3}}\right)\!\!\left(u^{\frac{m-n}{3}}yu^{\frac{-m+n}{3}}\right)\right)u^{\frac{-2m-n}{3}}\right)u^{m+n}  \\
             &= f^{ \frac{2m+n}{3}}\!\!\left(f^{\frac{-2m-n}{3}}(x)f^{\frac{m-n}{3}}(y)\right)u^{m+n}   
\end{alignat*}
for any $x,y\in N$.

\pagebreak

Here $G$ is a group if and only if $N$ is a group and $u$ associates with all of the elements of $N$.  
From the first part of the theorem, $x(yu^3)=f\!\left(f^{-1}(x)f^{-1}(y)\right)u^3$ for any $x,y\in N$.  Thus $u^3$, and therefore $u$, associates with all of the elements of $N$ if and only if $f\!\left(f^{-1}(x)f^{-1}(y)\right)u^3=(xy)u^3$ for any $x,y\in N$. Hence, $G$ is a group if and only if $N$ is a group and $f$ is an automorphism of $N$. \hfill $\square$

\begin{rem}
Note that Theorem \ref{T:main} generalizes to cases where $H$ is not necessarily finite of order coprime to three. Namely, if $G=NH$ is contained in a larger Moufang loop $K$ such that $N\trianglelefteq K$, $H=\left<u\right>\leq G$ and $u=v^3$ for some element $v\in K$ then  Equation (\ref{e:ext}) still holds where $f^{\frac{1}{3}}:G\longrightarrow G$ is defined by $f^{\frac{1}{3}}(g)=vgv^{-1}$.   
\end{rem}

\begin{rem}
From the proof of Theorem \ref{T:main} it can be seen that Equation (\ref{e:ext}) holds for any Moufang loop $G$ and any elements $x,y,u\in G$ where $u=v^3$ for some other element $v\in G$ and $f^{\frac{1}{3}}:G\longrightarrow G$ is the map $f^{\frac{1}{3}}(g)=vgv^{-1}$.   
\end{rem}

From Theorem \ref{T:main}, 
it follows 
that if $G = NH$ is a Moufang loop with $N \trianglelefteq G$ 
and $H = \left<u\right>\leq G$  of finite order coprime to three then the
extension depends only on 
how $u$ permutes the elements of $N$ as a semi-automorphism of~$N$.
Also observe that the converse of Theorem \ref{T:main} is not true.  That is, when taking a semi-automorphism of a Moufang loop $N$ and using (1) to construct a cyclic extension, the resulting loop is not always a Moufang loop.  Therefore, it would be interesting to find a larger class of loops that contains the class of Moufang loops and satisfies both Theorem \ref{T:main} and its converse.

It should also be noted that Theorem \ref{T:main} does not carry over to the case where $H$ has an order divisible by three.  It is known by H.~Zassenhaus~$\cite{81}$ along with T.~Kepka and P.~N\v{e}mec~$\cite{Kepka}$ that there are exactly two commutative Moufang loops of order 81 and exponent 3. One is an elementary abelian 3-group and the other of which is nonassociative.  Both can be written as $NH$ where $N$ is an elementary abelian 3-group of order 27 and $H$ is a cyclic group of order~3.  In each case $H$ acts trivially on $N$ proving that the extension depends on more~than just how the elements of $H$ permute the elements of $N$ by conjugation. 

Now suppose that $G=NH$ with $N$ and $H$ as above.  
If $N\cap H~=~1$ then the Moufang loop $G$ is a \it semidirect product \rm of $N$ and $H$ with a homomorphism $\varphi: H\rightarrow$ SemiAut$(N)$ mapping $u^n$ to $g\mapsto u^ngu^{-n}$.
Such a Moufang loop is isomorphic to the \it external semidirect product \rm induced by $\varphi$ 
\vspace{-.03in}
\[ N\rtimes_\varphi H=\left\{(x,u^m)\:|\:x\in N, u^m\in H\right\} \]
where $f=\varphi(u)$ and
\vspace{-.04in}
\begin{alignat}{1} 
(x,u^m)(y,u^n)&=\left(f^{ \frac{2m+n}{3}}\!\!\left(f^{\frac{-2m-n}{3}}(x)f^{\frac{m-n}{3}}(y)\right), u^{m+n}\right) \label{e:semi}
\end{alignat}
for any $(x,u^m),(y,u^n)\in N\rtimes_\varphi H$.

\begin{ex}\label{E:cheinloop}
A nice example of such an extension is a semidirect product $G=NH$, 
originally constructed by O.~Chein~$\cite{Chein}$, where $N$ is a~group~and $H=\left<u\right>$ is of order two.  
O. Chein defined the extension by extending the binary operation of $N$ to all of $G$ satisfying 
\vspace{-.03 in}
\begin{alignat*}{1}
(x u)y&=(xy^{-1})u \\
x(y u)&=(yx)u \\
(x u)(y u)&=y^{-1}x
\end{alignat*}
for any $x,y\in N$. Here one can picture the extension as $G=N \rtimes_\varphi H$ where 
\begin{alignat*}{3}
f=\varphi(u):\: &N &&\longrightarrow N\\
     &\:g &&\longmapsto ugu^{-1}
\end{alignat*}
is the semi-automorphism, namely the anti-automorphism, that inverts all of the elements of $N$.
\end{ex}

\begin{ex}  In $\cite{Rajah}$ A.\:Rajah proved that there exists an odd ordered non-associative Moufang loop of order $pq^3$ where 
$p$ and $q$ are distinct primes if and only if $q\equiv 1$~mod~$p$. For $q\equiv 1$~mod~$p$, A.~Rajah constructed a non-associative Moufang loop of order $pq^3$ with a nonabelian normal subgroup of order $q^3$ and exponent $q$. But, when $p\neq 3$, this is just an external semidirect product $N \rtimes_\varphi H$ where $H=\left<u\right>$ is of order $p$, $N=U_3({\mathbb F}_q)$ (the group of unipotent upper triangular matrices) 
and $\varphi(u)$ is a semi-automorphism of $N$ rather than an automorphism of $N$.  
Here $\varphi(u): N \rightarrow N$ is the map
\normalsize \[ \left(\begin{array}{lll}1 & a & c \\ 0 & 1 & b \\ 0 & 0 & 1\end{array}\right) \longmapsto \left(\begin{array}{lcc}1 & ak^{-1} & ck+ab\frac{k^{-2}-k}{2} \\ 0 & 1\:\:{}^{} & bk^{-1} \\ 0 & 0\:\:{}^{} & 1\:\:{}^{}\end{array}\right) \] \normalsize\it
where $k$ is a fixed element of  ${{}^{}}^{}{\mathbb F}_q$ with order $p$.
\end{ex}

\begin{thm}\label{T:same}
Let $N$ be a Moufang loop and $H$ be a finite cyclic group whose order is coprime to three.  Suppose that $\varphi_1,\varphi_2:H\rightarrow$ SemiAut$(N)$ are homomorphisms that create Moufang loops $G_1=N\rtimes_{\varphi_1}H$ and $G_2=N\rtimes_{\varphi_2}H$ as defined in Equation~(\ref{e:semi}). If $\varphi_2=\phi_\beta \circ \varphi_1 \circ \alpha$ where $\alpha\in$\:Aut$(H)$, $\beta\in$\:Aut$(N)$ and $\phi_\beta$ is conjugation by $\beta$ in SemiAut$(N)$ then the semidirect products $G_1$ and $G_2$ are isomorphic. 
\end{thm}

\begin{proof}
Let $f_a=\varphi_1(a)$ and $h_a=\varphi_2(a)$ for any $a\in H$ and let $H=\left<u\right>$. Note that 
\[ h_{\alpha^{-1}\!(u)}=\varphi_2(\alpha^{-1}\!(u))=\phi_\beta(\varphi_1(u))=\beta\circ f_u\circ\beta^{-1} \]
and thus $h_{\alpha^{-1}\!(u)}^k\circ\beta = \beta\circ f_u^k$ for any integer $k$.
Now define the bijection 
$$\psi~:~G_1~\longrightarrow~G_2$$ 
by 
$$\psi(x,u^m)=(\beta(x),\alpha^{\mbox{\tiny${-1}^{}$}}\!(u^m)).$$  From this it follows that, for any elements $(x,u^m),(y,u^n)\:\in\:G_1^{\frac{}{}}$,
\begin{alignat*}{1}
\psi&\!\left(x,u^m\right)\psi\!\left(y,u^n\right) = \left(\beta(x),\alpha^{-1}\!\!\left(u^m\right)\right)\left(\beta(y),\alpha^{-1}\!\!\left(u^n\right)\right) \\
  &= \left(\beta(x),\alpha^{-1}\!(u)^m\right)\left(\beta(y),\alpha^{-1}\!(u)^n\right) \\
  &= \left( h_{\alpha^{-1}\!(u)}^{\frac{2m+n}{3}}\!\!\left(h_{\alpha^{-1}\!(u)}^{\frac{-2m-n}{3}}\!(\beta(x))h_{\alpha^{-1}\!(u)}^{\frac{m-n}{3}}\!(\beta(y))\!\right),\alpha^{-1}\!(u)^m\alpha^{-1}\!(u)^n\right) \\
  &= \left( \!\left(\beta\circ\varphi_1(u)\circ\beta^{-1}\right)^{\!\!\frac{2m+n}{3}}\!\!\left(\!\left(\!\beta\circ f_u^{\frac{-2m-n}{3}}\!\right)\!\!(x)\!\left(\!\beta\circ f_u^{\frac{m-n}{3}}\!\right)\!\!(y)\!\right),\alpha^{-1}\!\!\left(u^{m+n}\right)\!\right) \\
  &= \left( \!\left(\beta\circ f_u^{\frac{2m+n}{3}}\circ\beta^{-1}\right)\!\!\left(\!\beta\!\left(f_u^{\frac{-2m-n}{3}}\!(x)f_u^{\frac{m-n}{3}}\!(y)\!\right)\!\right),\alpha^{-1}\!\!\left(u^{m+n}\right)\!\right) \\
  &= \left( \!\beta\!\left( f_u^{\frac{2m+n}{3}}\!\!\left(f_u^{\frac{-2m-n}{3}}\!(x)f_u^{\frac{m-n}{3}}\!(y)\right)\!\right),\alpha^{-1}\!\!\left(u^{m+n}\right)\!\right) \\
  &= \psi\!\left(f_u^{\frac{2m+n}{3}}\!\!\left(f_u^{\frac{-2m-n}{3}}\!(x)f_u^{\frac{m-n}{3}}\!(y)\right),u^{m+n}\right) \\
  &= \psi\!\big(\left(x,u^m\right)\left(y,u^n\right)\big).
\end{alignat*} \normalsize
Hence $N\rtimes_{\varphi_1}H\cong N\rtimes_{\varphi_2}H$.
\end{proof}

Note that Theorem 
$\ref{T:same}$ carries over to nonsplit extensions as long as $\beta|_{{}_{N\cap H}}=\alpha^{-1}|_{{}_{N\cap H}}$.

\begin{cor} 
 If the cyclic groups $\varphi_1(H)$ and $\varphi_2(H)$ are conjugate in the group SemiAut$(N)$ then $N\rtimes_{\varphi_1}H\cong N\rtimes_{\varphi_2}H$.
\end{cor}

 Here Theorems $\ref{T:main}$ and $\ref{T:same}$ are likely to stir up more interest in the study of semi-automorphisms.  It wasn't until 1999 that the conjecture made by F.~Dinkines, I.N.~Herstein and M.F.~Ruchte was solved for the case where the simple group is finite.
It was proven by K.I.~Beidar, Y.~Fong, W.F.~Ke and W.R.~Wu \cite{Beidar} that any nontrivial 
semi-endomorphism, and therefore \mbox{any semi-automorphism, of a} finite simple group 
is either an automorphism or an anti-automorphism. 
With this, the following theorem then follows from Theorem~$\ref{T:main}$.

\begin{thm}\label{T:simple group} Suppose $G$ is a finite Moufang loop with a normal subgroup $N$ that is simple.  If $G/N$ is cyclic with an order coprime to six then $G$ is a group. 
\end{thm}

A counter example 
for the case where $G/N$ has an order that is not coprime to six can be seen in Example $\ref{E:cheinloop}$, namely, the Chein loop.
It is then evident that this theorem raises the level of interest to the question asking whether or not there exists an infinite simple group with a semi-automorphism that is neither an automorphism nor an anti-automorphism.

Another question that arises is whether or not there exists a generalization of Theorem \ref{T:main} to better understand extensions of commutative Moufang loops \cite{abelian}.
Furthermore, Theorem \ref{T:main} may generate a result similar to \mbox{Theorem \ref{T:simple group} involving cyclic} extensions of \mbox{nonassociative  simple} Moufang loops. 

\end{document}